\documentclass[a4paper,10pt]{article}
\usepackage[utf8]{inputenc}

\usepackage[all]{xy}
\usepackage{amsmath,amssymb,amsthm}

\usepackage{graphicx}
\usepackage{color}
\usepackage[margin=1in]{geometry}
\usepackage[english]{babel} 

\usepackage{amsmath,amssymb}  
\usepackage{amsthm}           

\theoremstyle{plain}                          
 
\theoremstyle{definition}                     

\newtheorem{definition}{Definition}[section]
\newtheorem{theorem}{Theorem}[section]

\newtheorem{lemma}{Lemma}[section]    
\newtheorem{convention}{Convention}[section] 
\newtheorem{proposition}{Proposition}[section]
\newtheorem{corollary}{Corollary}[section]
\newtheorem{remark}{Remark}[section]

\theoremstyle{remark}                         

\newcommand{\field}[1]{\mathbb{#1}} 
\newcommand{\C}{\mathcal{C}} 
\newcommand{\Z}{\field{Z}} 
\newcommand{\K}{\field{K}}

\newcommand{\Q}{\field{Q}}
\renewcommand{\P}{\mathcal{P}}

\title{Algebraic Hopf Invariants and Rational Models for Mapping Spaces}
\author{Felix Wierstra}
\date{}

\begin{document}

\maketitle

\abstract{The main goal of this paper is to  define an invariant $mc_{\infty}(f)$ of homotopy classes of maps $f:X \rightarrow Y_{\Q}$, from a finite CW-complex $X$ to a rational space $Y_{\Q}$. We prove that this invariant is complete, i.e. $mc_{\infty}(f)=mc_{\infty}(g)$ if an only if $f$ and $g$ are homotopic. 

To construct this invariant we also construct a homotopy Lie algebra structure on certain convolution algebras. More precisely, given an operadic twisting morphism from a cooperad $\mathcal{C}$ to an operad $\mathcal{P}$, a $\C$-coalgebra $C$ and a $\P$-algebra $A$,  then there exists a natural homotopy Lie algebra structure on $Hom_\K(C,A)$, the set of linear maps from $C$ to $A$. We prove some of the basic properties of this convolution homotopy Lie algebra and use it to construct the algebraic Hopf invariants. This convolution homotopy Lie algebra also has the property that it can be used to models mapping spaces. More precisely, suppose that $C$ is a $C_\infty$-coalgebra model for a simply-connected finite CW-complex $X$ and $A$ an $L_\infty$-algebra model for a simply-connected rational space $Y_{\Q}$ of finite $\Q$-type, then $Hom_\K(C,A)$, the space of linear maps from $C$ to $A$, can be equipped with an $L_\infty$-structure such that it becomes a rational model for the based mapping space $Map_*(X,Y_\Q)$.


\tableofcontents

\section{Introduction}

One of the most elementary questions in algebraic topology is: "Given two maps $f,g:X \rightarrow Y$, are $f$ and $g$ homotopic?" In general this is an extremely hard question but a lot of partial progress has been made. One example is the Hopf invariant, originally this was an invariant of maps $f:S^{4n-1}\rightarrow S^{2n}$ and can for example be used to show that the Hopf fibration is not null homotopic. The Hopf invariant has been generalized by many people, (see for exampe \cite{SW2} and \cite{BT1} and their references for more details about this), but one of the most recent generalizations is due to Sinha and Walter (see \cite{SW2}). They generalize the Hopf invariant to maps $f:S^n \rightarrow Y_{\Q}$, from a sphere to a rational space $Y_{\Q}$ and prove that this generalization is a complete invariant, i.e. two maps $f$ and $g$ are homotopic if and only if their Hopf invariants coincide. 

The main goal of this paper is to generalize the work of Sinha and Walter to maps between arbitrary spaces. We do this by defining a generalization of their Hopf invariant, which is a function $mc_{\infty}:Map_*(X,Y_{\Q}) \rightarrow \mathcal{MC}(X,Y)$ from the set of pointed maps between a finite simply-connected CW complex $X$ and a simply-connected rational space $Y_{\Q}$, to the moduli space of Maurer-Cartan elements in a certain $L_{\infty}$-algebra associated to $X$ and $Y$. The main theorem of this paper is given by:

\begin{theorem}
 Let $f,g:X \rightarrow Y$ be two maps between a finite simply-connected CW-complex $X$ and a simply-connected space $Y$, then $f$ and $g$ are homotopic in the model category of rational spaces if and only if $mc_{\infty}(f)=mc_{\infty}(g)$.
\end{theorem}

To prove this theorem, we will first generalize the statement of the theorem and define a complete invariant of homotopy classes of maps between coalgebras over a cooperad $\mathcal{C}$. To do this it will be necessary to define an $L_{\infty}$-algebra structure on the convolution algebra between a $\mathcal{C}$-coalgebra and a $\Omega_{op} \mathcal{C}$-algebra, where $\Omega_{op} \mathcal{C}$ is the operadic cobar construction on the cooperad $\mathcal{C}$. We will do this in a more general setting and define an $L_{\infty}$-structure on the convolution algebra relative to an operadic twisting morphism in the following theorem.


\begin{theorem}\label{thrm11}
 Let $\tau:\mathcal{C}\rightarrow \mathcal{P}$ be an operadic twisting morphism, let $C$ be a $\mathcal{C}$-coalgebra and let  $L$ be a $\mathcal{P}$-algebra, then there exists a natural $L_{\infty}$-structure on the convolution algebra $Hom_{\K}(C,L)$.
\end{theorem}

This $L_{\infty}$-structure generalizes a construction from a paper by Dolgushev, Hoffnung and Rogers \cite{DHR1} and a construction from the book by Loday and Vallette (see \cite{LV}). In \cite{DHR1}, Dolgushev, Hoffnung and Rogers construct a special case of this $L_{\infty}$-structure for the canonical twisting morphism $\iota:\mathcal{C} \rightarrow \Omega_{op} \mathcal{C}$ and in \cite{LV}, Loday and Vallette define a Lie algebra relative to a binary quadratic twisting morphism. We will also give a shorter and more conceptual alternative proof for this theorem.

In \cite{RNW1}, the  $L_{\infty}$-structure from Theorem \ref{thrm11} is used  to prove that the convolution algebra becomes a rational model for the mapping space $Map_*(X,Y_{\Q})$. This is a generalization of a theorem by Berglund  about models for mapping spaces (see Theorem 1.4 of \cite{Berg1}).

\begin{theorem}[\cite{RNW1} Corollary 8.19]\label{thrm3}
 Let $C$ be a $C_{\infty}$-coalgebra model of finite type for a simply-connected CW-complex $X$ of finite $\Q$-type, such that $C$ is concentrated in degrees greater or equal than $2$. Let $L$ be an $L_{\infty}$-model for a simply-connected rational space $Y_{\Q}$ of finite $\Q$-type, then there exists an  $L_{\infty}$-structure on the space $Hom_{\K}(C,L)$, such that $Hom_{\K}(C,L)$ becomes a rational $SL_\infty$-model for  $Map_*(X,Y_{\Q})$. In particular, 
 \begin{enumerate}
\item  There is a bijection between $[X,Y_{\Q}]$ and $ \mathcal{MC}(Hom_{\K}(C,L))$ . 
\item For every Maurer-Cartan element $\tau$ we have an isomorphism 
$$\pi_n(Map_*(X,Y_{\Q}),\tau)\otimes \Q \simeq H_n(Hom_{\K}(C,L)^{\tau}).$$
 \end{enumerate}
Where $\mathcal{MC}(Hom_{\K}(C,L))$ denotes the moduli space of Maurer-Cartan elements in $Hom_\K(C,L)$. The homotopy groups are taken with respect to the base point given by the map corresponding to the Maurer-Cartan element  $\tau$. By $Hom_{\K}(C,L)^{\tau}$ we denote the $L_\infty$-algebra $Hom_\K(C,L)$ twisted by the Maurer-Cartan element $\tau$.


\end{theorem}

This theorem is an improvement of Theorem 6.3 form \cite{Berg1}. In the paper by Berglund this theorem assumes that $C$ is a strictly coassociative cocommutative coalgebra, we improve this theorem by relaxing the assumption that $C$ has to be coassociative. In our construction $C$ can also be a $C_{\infty}$-coalgebra which is a cocommutative coalgebra which is only coassociative up to homotopy. 

As a corollary of this theorem we give an alternative proof of a theorem by Buijs and Guti\'errez which states that  $Hom_{\K}(\tilde{H}_*(X;\Q),\pi_*(Y)\otimes \Q)$ can be equipped with an $L_{\infty}$-structure such that it becomes a model for the mapping space $Map_*(X,Y)$.

\begin{corollary}
 Under the assumptions of Theorem \ref{thrm3}, there exists an explicit $L_{\infty}$-structure on the space $Hom_{\K}(\tilde{H}_*(X;\Q),\pi_*(Y)\otimes \Q)$, such that this becomes a rational model for the mapping space $Map_*(X,Y_{\Q})$ as in Theorem \ref{thrm3}.
\end{corollary}

\subsection{Acknowledgments}

The author would like to thank Alexander Berglund and Dev Sinha for many useful conversations and comments on this paper. The author would also like to thank Philip Hackney for answering many of his questions.

\section{Conventions}

In this paper we will follow the following conventions. 

\begin{convention}
 Throughout this paper we will assume that $\K$ is a field of characteristic $0$, and in the topological part of the paper we will work over the rationals. Further we will assume that all homology and cohomology groups are taken with rational coefficients. For simplicity we will often omit this from the notation.
\end{convention}

\begin{convention}
In this paper we will denote the symmetric group on $n$ letters by $\Sigma_n$. We assume that all the operads and cooperads we consider in this paper are symmetric. We also assume that all operads and cooperads are connected , i.e. $\mathcal{P}(0)=0 $ and $\mathcal{P}(1)=\K$ for operads and $\mathcal{C}(0)=0$ and $\mathcal{C}(1)=\K$ for cooperads. 
\end{convention}

\begin{convention}
We will assume that all cooperads and coalgebras we consider are conilpotent. This means that the coradical filtration is exhaustive.                                                                                                                                                                                                                                                                                                                                                                                                                                                                                                                                                                                                                                                                                                                                                                                                                                                                                                                                                                                                                                                  See Section 5.8.4 and 5.8.5 in \cite{LV} for a definition and more details. 
\end{convention}


\begin{convention}
In this paper we will tacitly assume that all spaces are based and that all the mapping spaces are spaces of pointed maps. We also assume that all spaces except for the mapping spaces are $1$-reduced, i.e. they have only one zero-cell and no one-cells. In particular, all the spaces, except for the mapping spaces, are assumed to be simply-connected. Since we are studying based mapping spaces we will always work with reduced homology and cohomology, so the notation $\tilde{H}_*(X)$ will mean the reduced homology of the space $X$ with coefficients in $\Q$. As a consequence of this all algebras and coalgebras will be non-(co)unitial.

We will also assume that all homotopy groups are rational homotopy groups, to ease the notation we will simply write $\pi_*(X)$ instead of $\pi_*(X) \otimes \Q$ to denote the rational homotopy groups of a space $X$.
\end{convention}

\begin{definition}
 Let $X$ and $Y$ be based spaces, then we denote the space of based maps from $X$ to $Y$ by $Map_*(X,Y)$.
\end{definition}

\begin{convention}
 In this paper we will always work with $\Z$-graded chain complexes, which we grade homologically, so the differential will always have degree $-1$. We will often omit the dg, and unless stated otherwise we implicitly assume that all operads, cooperads, algebras and coalgebras are taken in the category of chain complexes. Further, recall that we also use the Koszul sign rule, i.e. the exchange map  $\tau:V \otimes W \rightarrow W \otimes V$ is given by $\tau(v\otimes w)=(-1)^{\vert v \vert \vert w \vert}w \otimes v$ and therefore introduces a sign $(-1)^{\vert v \vert \vert w \vert}$. The supsension f a chain complex $V$ is denoted by $sV$ and is defined by $(sV)_n=V_{n-1}$. The linear dual of a chain complex $V$ is denoted by $V^{\vee}$ and is defined as $V^\vee:=Hom_\K(V,\K)$.
 
 The only exception to our grading convention is the cohomology of a space which we will grade cohomologically. We will call a chain complex simply-connected if it is concentrated in degrees greater or equal than $2$. 
\end{convention}

\begin{definition}
The rational homotopy Lie algebra of a space $X$ is denoted by $\pi_*(X)$ and the linear dual of the rational homotopy groups is denoted by $\pi^*(X)$ and will be called the cohomotopy groups. 
\end{definition}

\begin{remark}
Note that our definition of the cohomotopy groups is different from the standard notion of cohomotopy groups. The cohomotopy group $\pi^{n}(X)$ is not the set of homotopy classes of maps $[X,S^n]$.
\end{remark}

\part{Preliminaries}

\section{Twisting morphisms, Koszul duality \\ and bar constructions}\label{sectwistingmorphism}

In this section we recall the most important concepts about operads, cooperads, twisting morphisms and the bar and cobar constructions.

\subsection{Convolution algebras and twisting morphisms}\label{subsecMCeq}

In this paper we use several different bar constructions, in this section we will briefly recall the definitions and introduce some notation. All our conventions and definitions are based on the book \cite{LV}, unless stated otherwise.

\begin{definition}
The operadic bar construction on an operad $\mathcal{P}$ is denoted by $B_{op} \mathcal{P}$ and the operadic cobar construction on a cooperad $\mathcal{C}$ is denoted by $\Omega_{op} \mathcal{C}$.
\end{definition}

\begin{definition}
Let $\alpha:\mathcal{C} \rightarrow \mathcal{P}$ be an operadic twisting morphism, $(C,\Delta_C,d_C)$ a $\mathcal{C}$-coalgebra and $(A,\mu_A,d_A)$ a $\mathcal{P}$-algebra. The convolution algebra $Hom_{\K}(C,A)$ is the dg vector space of all linear maps from $C$ to $A$. The differential is defined by 
$$\partial(f)=d_A \circ f - (-1)^{\mid f \mid} f \circ d_C.$$
The Maurer-Cartan operator $\star_{\alpha}:Hom_{\K}(C,A) \rightarrow Hom_{\K}(C,A)$ is defined by 
$$\star_{\alpha}(f)=C \xrightarrow{\Delta_C} \mathcal{C} \circ C \xrightarrow{\alpha \circ f} \mathcal{P} \circ A \xrightarrow{\mu_A} A.$$ 
\end{definition}

\begin{definition}
A twisting morphism relative to an operadic twisting  morphism  $\alpha:\C \rightarrow \P$ is defined as a degree $0$ linear map $\tau:C \rightarrow A$, such that $\tau$ satisfies the Maurer-Cartan equation, which is given by 
$$\partial(\tau)+\star_{\alpha}(\tau)=0.$$
The set of Maurer-Cartan elements relative to $\alpha$ is denoted by $MC_\alpha(C,A)$.
\end{definition}

\subsection{Bar and cobar constructions for algebras over an operad}

The main reason we care about twisting morphisms in this paper, is because they are represented by the bar and cobar constructions. The bar construction will provide us with functorial fibrant replacements in the model category of $\mathcal{C}$-coalgebras and the cobar  construction will provide us with functorial cofibrant replacements in the category of $\mathcal{P}$-algebras. The bar and cobar constructions also give us a way to relate the  rational homology and rational homotopy groups to each other (see for example \cite{Ber1}). In this section we define the bar and cobar construction for algebras and coalgebras relative to an operadic twisting morphism $\alpha:\mathcal{C} \rightarrow \mathcal{P}$, between a cooperad $\mathcal{C}$ and an operad $\mathcal{P}$. The bar and cobar construction form an adjoint pair of functors between the category of $\mathcal{C}$-coalgebras and $\mathcal{P}$-algebras:
$$\Omega_{\alpha}:\{\mbox{conilpotent $\mathcal{C}$-coalgebras}\} \rightleftarrows \{\mbox{$\mathcal{P}$-algebras} \}:B_{\alpha}.$$
The constructions are given in the following definitions. 
 
\begin{definition}
Let $A$ be a $\mathcal{P}$-algebra and let $\alpha:\C \rightarrow \P$ be an operadic twisting morphism. The bar construction $B_{\alpha}A$ on $A$ is defined as $B_{\alpha}A=(\mathcal{C}(A),d_B)$. Where $\mathcal{C}(A)$ denotes the cofree $\mathcal{C}$-coalgebra on the underlying vector space of $A$ and $d_B=d_1+d_2$. The differential $d_1$ is the unique extension of $d_A$ and $d_\C$ to $\mathcal{C}(A)$ and $d_2$ is the unique coderivation extending the following map
$$\mathcal{C} \circ A \xrightarrow{\alpha \circ Id_A} \mathcal{P} \circ A \xrightarrow{\mu_A} A.$$ 
\end{definition}

Similarly we also have the cobar construction which is defined as follows.

\begin{definition}
Let $C$ be a $\mathcal{C}$-coalgebra, then we define the cobar construction $\Omega_{\alpha}C=(\mathcal{P}(C),d_{\Omega})$ as the free $\mathcal{P}$-algebra on the underlying vector space of $C$. The differential is given by $d_{\Omega}=d_1+d_2$, where $d_1$ is the unique extension of $d_C$ and $d_{\P}$ to $\mathcal{P}(A)$ and $d_2$ is the unique derivation that extends the following map 
$$C \xrightarrow{\Delta_C} \mathcal{C} \circ C \xrightarrow{\alpha \circ Id_C} \mathcal{P} \circ C.$$ 
\end{definition}

The following theorem states that the set of Maurer-Cartan elements are represented by the bar and cobar construction.

\begin{theorem}
Let $\alpha:\mathcal{C} \rightarrow \mathcal{P}$ be an operadic twisting morphism, $C$ a $\mathcal{C}$-coalgebra and $A$ a $\mathcal{P}$-algebra. Then there are natural bijections 
$$Hom_{\mathcal{P}-alg}(\Omega_{\alpha}C,A) \cong MC_{\alpha}(C,A) \cong Hom_{\mathcal{C}-coalg}(C,B_{\alpha}A).$$
\end{theorem}

For the proof see Proposition 11.3.1 in \cite{LV}.

\subsection{Universal twisting morphisms}

The bar and cobar construction have the universal property that every twisting morphism factors uniquely through two universal twisting morphisms associated to the bar and cobar construction. For more details see Section 6.5.11 in \cite{LV}.


\begin{proposition}\label{prop1234}
Let $\tau:C \rightarrow A$ be a twisting morphism relative to an operadic twisting morphism $\alpha:\mathcal{C} \rightarrow \mathcal{P}$. Then there exist universal twisting morphisms $\pi:B_{\alpha} A \rightarrow A$ given by the projection of $B_{\alpha}A$ onto $A$ and $\iota:C \rightarrow \Omega_{\alpha} C$ given by the inclusion of $C$ into $\Omega_{\alpha} C$.  Such that $\tau$ factors uniquely through $\pi$ in the sense that there exists a unique map of $\mathcal{C}$-coalgebras $f_{\tau}:C \rightarrow B_{\alpha}A$ such that $\tau=\pi \circ f_{\tau}$. Similarly, there exists a unique map of $\mathcal{P}$-algebras $g_{\tau}:\Omega_{\alpha} C \rightarrow A$, such that $\tau = g_{\tau} \circ \iota$. This is summarized in the following diagram

$$
\xymatrix{
{} & B_{\alpha} A  \ar[rd]^{\pi}& \\
C \ar[rr]^{\tau} \ar[dr]^{\iota} \ar[ru]^{f_{\tau}} & & A \\
& \Omega_{\alpha} C \ar[ru]^{g_{\tau}}  &
}
$$



\end{proposition}


The main reason to consider the bar and cobar construction is because they give functorial fibrant and cofibrant replacements, this will be explained in Section \ref{secmodelcategories}, the following theorem states that the bar-cobar resolution and the cobar-bar resolution are in fact resolutions of  algebras and coalgebras. The following theorem is a combination of Theorem 11.3.3 and Theorem 11.3.4 of \cite{LV}.

\begin{theorem}
If the operadic twisting morphism $\alpha:\C \rightarrow \P$ is Koszul, then the counit $\epsilon_{\alpha}:\Omega_{\alpha} B_{\alpha} A \rightarrow A $  and unit $\nu_{\alpha}: C \rightarrow B_{\alpha} \Omega_{\alpha} C $ of the bar-cobar adjunction are quasi-isomorphisms. 
\end{theorem}

\section{The $L_{\infty}$-operad and $L_{\infty}$-algebras}

One of the most important operads in this paper will be the $L_{\infty}$-operad. The $L_{\infty}$-operad is the operad describing Lie algebras up to homotopy and is heavily used in for example deformation theory and rational homotopy theory. In this section we recall most of the basics of the theory of $L_{\infty}$-algebras, for more details see for example \cite{Getz1}, \cite{BM11} and \cite{Berg1}.


\begin{remark}
In this paper we will use shifted $L_\infty$-algebras, which we denote by $SL_\infty$. It turns out that that shifted $L_\infty$-algebras have several technical advantages compared to unshifted $L_\infty$-algebras. The reader who prefers to use the unshifted version just needs to (de)suspend everything. The main difference between $L_\infty$-algebras and $SL_{\infty}$-algebras is that the generating operations of the $SL_{\infty}$-operad all have degree $-1$. In the unshifted case the generating operation of arity $n$ has degree $2-n$.

\end{remark}

\begin{definition}
The cocommutative cooperad $\mathcal{COCOM}$ is the cooperad given by the symmetric sequence $\mathcal{COCOM}(k)=\K \mu_k$ which is one-dimensional in each arity and has the trivial symmetric group action, all the cooperations $\mu_k$ have degree $0$. The decomposition map is given by 
$$\Delta_{\mathcal{COCOM}}(\mu_n)=\sum_{p=1}^{n} \sum_{i=1}^{p} \mu_p \circ_i \mu_{n-p+1}.$$
\end{definition}

\begin{definition}
The $SL_{\infty}$-operad is the operad given by $\Omega_{op}\mathcal{COCOM}$, the operadic cobar construction on the cooperad $\mathcal{COCOM}$, i.e. it is the free operad generated by $s^{-1}\mathcal{COCOM}$. The differential comes from the cobar construction and is easiest described as the higher Jacobi identities on an $SL_\infty$-algebra. 

 An $SL_{\infty}$-algebra $L$ is an algebra over the $SL_{\infty}$-operad, more specifically it is a differential graded vector space $L$ with a sequence of  operations $l_n:L^{\otimes n} \rightarrow L$  for each $n \geq 2$. All the operations $l_n$ have dgree $-1$, are graded symmetric and satisfy a shifted version of the higher Jacobi identities. These higher Jacobi identities are given by
\[
	\sum_{\substack{n_1+n_2 = n+1\\\sigma\in Sh(n_2,n_1-1)}}(l_{n_1}\circ_1l_{n_2})^\sigma = 0,
\]
where $Sh(n_2,n_1-1)$ denotes the set of shuffles between the sets $\{1,...,n_2\}$ and $\{1,...,n_1-1\}$. The differential of $L$ is here denoted by $l_1$.
\end{definition}



In this paper we will often restrict ourselves to two special classes of $L_\infty$-algebras called nilpotent and degree-wise nilpotent $L_\infty$-algebras. These classes of $SL_\infty$-algebras have better properties and behave better with respect to certain infinite sums that we need to take. For more details see Section 2 of \cite{Berg1}.

\begin{definition}\label{defnilpotence}
Let $L$ be an $SL_\infty$-algebra, the lower central series of $L$ is defined as the descending filtration $\Gamma_1 \supseteq ... \supseteq \Gamma_{n-1} \supseteq \Gamma_n L\supseteq ...$, whose $n$th piece $\Gamma_n L$, is spanned by all possible bracket expressions using at least $n$ elements of $L$.

An $SL_\infty$-algebra $L$ is called nilpotent if there exists an $N$ such that $\Gamma_n L =0$ for all $n >N$. 

An $SL_\infty$-algebra $L$ is called degree-wise nilpotent if for each $d$ there exists an $N_d$ such that $(\Gamma_n L)_d=0$ for all $n > N_d$.
\end{definition}

Note that nilpotence clearly implies degree-wise nilpotence. In general we can always equip the tensor product of an $SL_\infty$-algebra and a commutative algebra with the structure of an $SL_\infty$-algebra. We call this the extension of scalars of $L$ by $A$. This is defined as follows.

\begin{definition}
 Let $(A,\mu)$ be a commutative algebra and $(L,l_1,l_2,...)$ an $L_{\infty}$-algebra, then we define the extension of scalars of $L$ by $A$, as the $L_{\infty}$-algebra whose underlying chain complex is $A \otimes L$. The differential is given by $d_{A \otimes L}(a \otimes x)=d_A(a)\otimes x+ (-1)^{\vert a \vert} a\otimes d_L(x)$, for $a \otimes x \in A \otimes L$. The operations $l_n$, for $n \geq 2$ are given by 
 \[
 l_n(a_1\otimes x_1,...,a_n \otimes l_n)=(-1)^{\sum_{i<j}\vert a_i\vert \vert x_j \vert} a_1 \cdot a_2 \cdot ... \cdot a_{n-1} \cdot a_n \otimes l_n(x_1,...,x_n),
 \]
where $a_i \otimes x_i \in A \otimes L$ and $a_1 \cdot a_2 \cdot ... \cdot a_{n-1} \cdot a_n$ denotes the  product of $A$.
\end{definition}

To each degree-wise nilpotent $SL_{\infty}$-algebra we can associate a simplicial set as follows, see \cite{Getz1} for more details.

\begin{definition}
 Let $L$ be a degree-wise nilpotent $SL_{\infty}$-algebra, a Maurer-Cartan element $\tau\in L$  is a degree $0$ element satisfying the Maurer-Cartan equation which is given by
 $$\sum_{n \geq 1} \frac{1}{n!} l_n(\tau,...,\tau)=0.$$
 
 The set of Maurer-Cartan elements in an $SL_{\infty}$-algebra is denoted by $MC(L)$.
\end{definition}  

We need that the $SL_\infty$-algebra $L$ is degree-wise nilpotent to make sure that the Maurer-Cartan equation converges. One of the special things about Maurer-Cartan elements is that we can use them to twist  $L$, that is we can define a new $SL_\infty$-algebra  structure $L^{\tau}$, on $L$ using a Maurer-Cartan element $\tau$. See Proposition 4.4 of \cite{Getz1} for more details.
  
\begin{proposition}
Let $L$ be a degree-wise nilpotent $SL_\infty$-algebra and let $\tau\in L$ be a Maurer-Cartan element. Then there exists a new $SL_\infty$-algebra structure on $L$, which we denote by $L^{\tau}$. For each $n \geq 1$, the twisted $SL_\infty$-operation $l_n$ on $L^\tau$ is given by 
\[
l_n^{\tau}(x_1,...,x_n)=\sum_{k \geq 0}\frac{1}{k!} l_{n+k}(\tau,...,\tau,x_1,...,x_n),
\]
where $x_i \in L$ and the element $\tau$ appears $k$ times.  
\end{proposition}

For the next definition recall that $\Omega_n$, the polynomial de Rham forms on the $n$-simplex, are defined as the commutative algebra given by 
\[
\Omega_n= \Lambda[t_0,...,t_n,dt_0,...,dt_n]/ (T_n,dT_n),
\] 
where $\Lambda$ is the free graded commutative algebra, the variables $t_i$ have degree $0$ and the variables $dt_i$ have degree $1$, the differential is given by $d(t_i)=dt_i$ Note that we have graded this algebra cohomologically. The relations $T_n$ and $dT_n$ are given by $T_n=t_0+...+t_n-1$ and $dT_n=dt_0+...+dt_n$.

\begin{definition}
  To each degree-wise nilpotent $SL_{\infty}$-algebra $L$ we associate a simplicial set $MC_{\bullet}(L)$, whose set of $n$-simplices is given by $MC(L \otimes \Omega_n)$, where $\Omega_n$ is the commutative algebra of polynomial de Rham forms on the $n$-simplex. The face and degeneracy maps are the ones induced by the face and degeneracy maps of $\Omega_{\bullet}$.
\end{definition}

The following theorem comes from \cite{Getz1}.

\begin{theorem}
 The simplicial set $MC_{\bullet}(L)$ is a Kan complex.
\end{theorem}

Using this definition we can now define rational $SL_\infty$-models for a space $X$.

\begin{definition}
Let $X$ be a simply-connected space of finite $\Q$-type, a rational $SL_{\infty}$-model $L$ for the space $X$ is an $SL_{\infty}$-algebra $L$, such that $MC_{\bullet}(L)$ is rationally equivalent to the space $X$.
\end{definition}

In \cite{Getz1}, it is shown that rational $SL_{\infty}$-models for simply-connected spaces of finite $\Q$-type exist.

In this paper we are mainly interested in the set of path components of the simplicial set $MC_{\bullet}(L)$.  Because the simplicial set $MC_\bullet(L)$ is a Kan complex, the equivalence relation of being in the same path component of $MC_\bullet(L)$ defines an equivalence relation on $MC(L)$. This equivalence relation will be called homotopy or gauge equivalence. In particular two Maurer-Cartan elements $x, y \in L$ are gauge equivalent if there exists a Maurer-Cartan element $z \in L \otimes \Omega_1$ such that $z$ is of the form $z=\sum z_i \otimes P_i(t) + z'_i \otimes Q_i(t)dt$, where the sum runs over a basis $\{z_i\}$ of $L$ and $P$ and $Q$ are polynomials in $t$. The elements $x$ and $y$ are then called gauge equivalent if $\sum z_i \otimes P_i(0)=x$ and $\sum z_i \otimes P_i(1)=y$.


\begin{definition}
 The moduli space of Maurer-Cartan elements in an $SL_{\infty}$-algebra $L$, is defined as the set of Maurer-Cartan elements modulo the gauge equivalence relation. The moduli space of Maurer-Cartan elements is denoted by $\mathcal{MC}(L)$.
\end{definition}

\begin{remark}
In deformation theory it is common to have an alternative definition of gauge equivalence given by the action of the "Lie group" associated to the $SL_{\infty}$-algebra on the set of Maurer-Cartan elements, in the papers \cite{BM11} and \cite{DP1} it is proven that gauge and homotopy equivalence are the same equivalence relation.  Since the word homotopy equivalence is already quite overused in this paper, we shall refer to this equivalence relation as gauge equivalence.
\end{remark}

\subsection{$C_{\infty}$-models for spaces}

In the previous section we have described the $SL_{\infty}$ approach to rational homotopy theory, but there is also a second approach with $C_{\infty}$-coalgebras. In this section we will recall the basic facts about this approach. We start by  defining the $C_{\infty}$-cooperad, this is a fibrant replacement of the cocommutative cooperad.

\begin{definition}
The $C_{\infty}$-cooperad is the cooperad defined as the bar construction on the operad $\mathcal{LIE}$, i.e. $C_{\infty}=B_{op} \mathcal{LIE}$.
\end{definition}

The next step is to use the following theorem by Quillen, which is the main theorem of \cite{Quil1}. Recall that we call a chain complex simply-connected if it is concentrated in degrees greater or equal than $2$.

\begin{theorem}\label{thrmQuillenstheorem}
 There exists a functor $\mathcal{C}\lambda:Top_{*,1} \rightarrow CDGC_{\geq 2}$ which induces and equivalence of homotopy categories between the category of 1-reduced rational spaces of finite $\Q$-type to the category of simply-connected cocommutative differential graded coalgebras.
\end{theorem}

Using this theorem we define a $C_{\infty}$-model for a space as follows.

\begin{definition}
Let $X$ be a simply-connected space of finite $\Q$-type, a $C_{\infty}$-coalgebra $C$ is a  $C_{\infty}$-model for $X$ if there exists a zig-zag of quasi-isomorphisms between $\mathcal{C} \lambda (X)$ and $C$. 
\end{definition}

\subsection{The convolution operad}

In this section we recall the definition of the convolution operad. This is an operad associated to  $Hom_{\K}(\mathcal{C},\mathcal{P})$, the space of linear maps between a cooperad $\mathcal{C}$ and an operad $\mathcal{P}$. The convolution operad was first defined in \cite{BM}. For more details see also \cite{LV} on which this section is based.

\begin{definition}
Let $\mathcal{C}$ be a cooperad and let $\mathcal{P}$ be an operad, then we define a operad structure on  $Hom_{\K}(\mathcal{C},\mathcal{P})$, the space of linear maps from $\C(n)$ to $\P(n)$, as follows. 

The arity $n$ part of the convolution operad is defined as $Hom_\K(\C,\P)(n):=Hom_{\K}(\mathcal{C}(n),\mathcal{P}(n))$. The symmetric group action on  $Hom_{\K}(\mathcal{C},\mathcal{P})(n)$ is defined by 
$$f^{\sigma}(x)=\sigma (f(x^{\sigma^{-1}} )),$$
for $x \in \C$, $\sigma \in  \Sigma_n$ and where $x^{\sigma^{-1}}$ denotes the action of $\sigma^{-1}$ on $x$. Let $f \in Hom(\C,\P)(k)$ and $g_i\in Hom(\C,\P)(n_i)$ for $i=1,...,k$ with $i_1+...i+k=n$, then we define the composition $f \circ g_1,...,g_k$ by the following composite of maps
\[
\C(n) \xrightarrow{\Delta_{\C}} (\C \circ \C)(n) \rightarrow \C(k) \otimes \C(i_1) \otimes... \otimes  \C(i_k) \otimes \K[\Sigma_n]\xrightarrow{f \otimes g_1 \otimes ... \otimes g_k \otimes Id_{\Sigma_n}}
\]
\[
 \P(k) \otimes \P(i_1) \otimes... \otimes \P(i_k) \otimes \K[\Sigma_n] \xrightarrow{\gamma_{\P}} \P(n),
\]
where $\Delta_\C$ is the decomposition map of $\C$ and $\gamma_\P$ is the composition map of $\P$.
\end{definition}

It turns out that the cocommutative cooperad plays a special role, since it will act as a "unit" with respect to the convolution operad. The following lemma will be important in section \ref{seclinfty}.

\begin{lemma}\label{lemcocomunit}
Let $\P$ be an operad, there exists a canonical isomorphism of operads between $\P$ and  $Hom_\K(\mathcal{COCOM},\P)$.
\end{lemma}

\begin{proof}
We have to show that there is an isomorphism between $Hom_{\K}(\mathcal{COCOM},\mathcal{P})$ and $\mathcal{P}$. To do this we define an explicit isomorphism 
$$\phi:Hom_{\K}(\mathcal{COCOM},\mathcal{P})\rightarrow \mathcal{P},$$ 
which is given by sending a map $f:\mathcal{COCOM} \rightarrow \mathcal{P}$ to its image in $\mathcal{P}$, i.e. in arity $n$ component of this map is given by
$$\phi(f)=f(\mu_n).$$
Where $\mu_n$ is the basis element of $\mathcal{COCOM}(n)$. Since $\mathcal{COCOM}$ is one-dimensional in each arity, the map $\phi$ is an aritywise isomorphism of vector spaces. The morphism $\phi$ commutes with the symmetric group action because the symmetric group action on $\mathcal{COCOM}$ is trivial. Therefore the action on a map $f \in Hom_{\K}(\mathcal{COCOM},\mathcal{P})$ is given by $\sigma(f)$, which is the same as the action coming from $\mathcal{P}$. It is a straightforward but tedious check that the morphism $\phi$ commutes with the operadic composition maps and we leave this to the reader.

\end{proof}

\subsection{The pre-Lie algebra associated to an operad and operadic twisting morphisms}

The goal of this section is to recall how we can associate a pre-Lie algebra to an operad and how the Maurer-Cartan elements in the pre-Lie algebra associated to the convolution operad correspond to operadic twisting morphisms. See \cite{LV} for more details.

\begin{proposition}
Let $\mathcal{P}$ be an operad, then there exists a pre-Lie algebra structure on the space $\bigoplus_{n \geq 1} \mathcal{P}(n)$ with the pre-Lie operation given by 
$$\mu \star \nu=\sum_{i=1}^{n} \sum_{P} (\mu \circ_{i} \nu )^{\sigma_{P}},$$
with $\mu \in \mathcal{P}(n)$ and $\nu \in \mathcal{P}(m)$ and the second sum runs over all ordered partitions $P \in Ord(1,...,1,n-i+1,1,...,1)$, with $n-i+1$ on the $i$th spot.
\end{proposition}

In a pre-Lie algebra we also have a notion of Maurer-Cartan equation and Maurer-Cartan elements. The Maurer-Cartan equation is given by $d \alpha + \alpha \star \alpha=0$ and $\alpha$ is called a Maurer-Cartan element if it is of degree $-1$ and satisfies the Maurer-Cartan equation.

\begin{remark}
Note that we are not working with shifted pre-Lie algebras, the Maurer-Cartan elements in a pre-Lie algebra therefore have degree $-1$ instead of $0$ (which is the degree of Maurer-Cartan elements in the $SL_\infty$ case).
\end{remark} 

\begin{definition}
Let $\C$ be a cooperad and $\P$ be an operad. A linear map $\alpha:\C \rightarrow \P$ is called a twisting morphism if it is a Maurer-Cartan element in the pre-Lie algebra associated to  $Hom_\K(\C,\P)$. The set of twisting morphisms is denoted by $Tw(\C,\P)$.
\end{definition}

Twisting morphisms have the following property, which can be found as Theorem 6.5.7 in \cite{LV}

\begin{theorem}\label{thrm6.5.7}
Let $\C$ be a cooperad and let $\P$ be an operad then we have the following bijections
\[
Hom_{Cooperads}(\C,B_{op}\P) \cong Tw(\C,\P) \cong Hom_{Operads}(\Omega_{op}\C,\P).
\]
\end{theorem}

\section{Model structures on algebras and coalgebras}\label{secmodelcategories}

In this paper we will use several different model categories, in this section we introduce the model categories we will use. For an introduction to the theory of model categories we recommend \cite{DS1}, from which we have taken most of the definitions and conventions.

First we will define a model structure on the category of $\mathcal{P}$-algebras, which will induce a model structure on the category of $\mathcal{C}$-coalgebras. This was originally done by Hinich in \cite{Hin1}.

\begin{theorem}
Let $\mathcal{P}$ be an operad, a model structure on the category of $\mathcal{P}$-algebras is given by: 
\begin{itemize}
\item The weak equivalences are given by the quasi-isomorphisms.
\item The fibrations are given by maps that are degree-wise surjective.
\item The cofibrations are the maps with the left lifting property with respect to acyclic fibrations.
\end{itemize}
\end{theorem}

The following theorem is Theorem 3.11 in \cite{DCH1} and given a twisting morphism $\tau:\C \rightarrow \P$, it provides us with a model structure on the category of conilpotent coalgebras over a cooperad $\mathcal{C}$ relative to the twisiting morphism $\tau:\C \rightarrow \P$. This model structure was originally defined by Vallette in \cite{Val1}, in the case that the twisting morphism is Koszul.

\begin{theorem}
Let $\mathcal{C}$ be a cooperad and let $\tau:\mathcal{C} \rightarrow \mathcal{P}$ be an operadic twisting morphism to an operad $\mathcal{P}$. Then there exists a model structure on the category of conilpotent-$\mathcal{C}$ coalgebras such that,
\begin{itemize}
\item The weak equivalences are created by the cobar construction, i.e. a map $f:C \rightarrow D$ is a weak equivalence if  $\Omega_{\tau } f: \Omega_{\tau} C \rightarrow \Omega_{\tau} D$ is a quasi-isomorphism of $\mathcal{P}$-algebras.
\item The cofibrations are the morphisms $f:C \rightarrow D$, such that $f$ is a degree-wise monomorphism.
\item The fibrations are the morphisms with the right lifting property with respect to the acyclic cofibrations.
\end{itemize}
\end{theorem}

The following theorem is Proposition 3.15 from \cite{DCH1}.

\begin{theorem}\label{thrmquilleneqcoalg}
Let $\tau:\mathcal{C} \rightarrow \mathcal{P}$ be an operadic twisting morphism, assume that $\mathcal{C}$ and $\mathcal{P}$ are weight graded and that $\tau$ respects the weight grading. In this case there is a Quillen equivalence between the categories of $\mathcal{P}$-algebras and $\mathcal{C}$-coalgebras if and only if the twisting morphism $\tau$ is Koszul.
\end{theorem}



 In the following proposition we will describe the fibrant and cofibrant objects in the model categories of algebras and coalgebras.

\begin{proposition}\label{propfibrantobjects}
In the model category of algebras over an operad $\mathcal{P}$ every object is fibrant, the cofibrant objects are retracts of quasi-free algebras $(\mathcal{P}(V),d)$ equipped with an exhaustive filtration on $V$,i.e. there is a filtration of the form
$$V_0=\{0\} \subseteq V_1 \subseteq ... \subseteq Colim_i V_i=V,$$
such that $d(V_i) \subseteq \mathcal{P}(V_{i-1})$.
In the model category of coalgebras over a cooperad $\mathcal{C}$ every object is cofibrant, the fibrant objects are given by the quasi-free $\mathcal{C}$-coalgebras. 
 
 \end{proposition}

 \begin{proof}
 For the proof of the algebra case see \cite{Hin1} and for the proof of the coalgebra case see Theorem 2.1 in \cite{Val1}. 
 \end{proof}

 \begin{remark}
  From Proposition \ref{propfibrantobjects} it follows that in the algebra case, $\Omega_{\tau} C$, the cobar construction on a $\mathcal{C}$-coalgebra $C$ with respect to a Koszul twisting morphism $\tau: \mathcal{C} \rightarrow \mathcal{P}$, is a cofibrant object. Similarly in the coalgebra case, all coalgebras of the form $B_{\tau} A$, for some $\mathcal{P}$-algebra $A$ are fibrant. Therefore we can define functorial cofibrant replacements by applying the cobar-bar resolution in the algebra case. In the coalgebra case we have a functorial fibrant replacement given by the bar-cobar resolution.
 \end{remark}

The following two lemmas will be important in section \ref{seccompleteness} to prove the completeness of the Hopf invariants. The first lemma is Lemma 4.9 in \cite{DS1} and the second lemma is Ken Brown's Lemma and can be found as Lemma 9.9 in \cite{DS1}.

\begin{lemma}\label{lemacyclicfibrations}
Let $A$ be a cofibrant object in a model category $\mathbf{C}$ and let $p:Y\rightarrow X$ be an acyclic fibration. Then composition with $p$ induces a bijection
$$p_*:\pi^{l}(A,Y) \rightarrow \pi^l(A,X),$$
where $\pi^l(A,Y)$ is the set of left homotopy classes of maps between $A$ and $Y$.
\end{lemma}

\begin{lemma}\label{lemkbrown}
Let $F:\mathbf{C}\rightarrow \mathbf{D}$ be a functor from a model category $\mathbf{C}$ to a model category $\mathbf{D}$, such that $F $ carries acyclic cofibrations between cofibrant objects to weak equivalences, then $F$ preserves all weak equivalences between cofibrant objects.
\end{lemma}

\section{The Homotopy Transfer Theorem, $\mathcal{P}_{\infty}$-algebras and $\mathcal{C}_{\infty}$-coalgebras}

The final ingredient we need to define the Hopf invariants is the Homotopy Transfer Theorem and the notion of algebras and coalgebras up to homotopy. In this paper we will use slightly different definitions for $\mathcal{P}_{\infty}$-algebras than in the book by Loday and Vallette \cite{LV}, instead of defining a $\mathcal{P}_{\infty}$-algebra as an algebra over the Koszul resolution $\Omega_{op} \mathcal{P}^{\text{!`}}$, we will define a $\mathcal{P}_{\infty}$-algebra as an algebra over the cobar-bar resolution $\Omega_{op} B_{op} \mathcal{P}$. This has the advantage that this resolution always exists and does not require the operad $\mathcal{P}$ to be Koszul. Because the categories of $\Omega_{op} \mathcal{P}^{\text{!`}}$-algebras and  $\Omega_{op} B_{op} \mathcal{P}$-algebras are Quillen equivalent (see Theorem 12.5A in \cite{Fres4}), it does not matter from a homotopy theoretical perspective if we would work with $\Omega_{op} \mathcal{P}^{\text{!`}}$-algebras or with $\Omega_{op} B_{op} \mathcal{P}$-algebras. The main disadvantage is that all the formulas will become more complicated than necessary, but because we will mainly use the Homotopy Transfer Theorem as a theoretical tool to prove completeness of the algebraic Hopf invariant, this is not a serious disadvantage. Altough the results here are presented in a  slightly different way than in \cite{LV}, all the proofs and details can be found in \cite{LV} or are completely analogous to the proofs and details there. For other papers on the Homotopy Transfer Theorem and homotopy algebras  see also \cite{Berg2} and \cite{DHR1}.


There are several equivalent definitions for $\mathcal{P}_{\infty}$-algebras.

\begin{definition}
A $\widetilde{\mathcal{P}_{\infty}}$-algebra is an algebra over $\Omega_{op} B_{op} \mathcal{P}$, the cobar-bar resolution of $\mathcal{P}$.
\end{definition}

\begin{remark}
To distinguish our notion of $\P$-algebra up to homotopy from the one defined in \cite{LV}, we use a tilde.
\end{remark}

The following theorem gives four alternative ways of describing $\widetilde{\mathcal{P}_{\infty}}$-structures on a graded vector space $A$, it can be found in \cite{LV} as Theorem 10.1.22. 

\begin{theorem}
 A $\widetilde{\mathcal{P}_{\infty}}$-algebra structure on a dg vector space $A$ can be described in the following ways:
 \begin{enumerate}
  \item By a morphism of operads $f:\Omega_{op} B_{op} \mathcal{P} \rightarrow End_A$, where $End_A$ is the endomorphism operad of the dg vector space $A$.
  \item By an operadic twisting morphism $\tau:B_{op} \mathcal{P} \rightarrow End_A$.
  \item By a morphism of cooperads $g:B_{op} \mathcal{P} \rightarrow B_{op} End_A$
  \item Or by a square zero coderivation on $B_{op} \mathcal{P}(A)$, the free $B_{op} \mathcal{P}$-coalgebra cogenerated by $A$.
 \end{enumerate}
Therefore we have bijections between the following sets which all describe the set of $\widetilde{\mathcal{P}_{\infty}}$-structures,
$$Hom_{Operads}(\Omega_{op}B_{op}\mathcal{P},End_A)\cong Tw(B \mathcal{P},End_A) \cong$$
$$Hom_{Cooperads}(B_{op} \mathcal{P} , B_{op} End_A) \cong Codiff(B_{op} \mathcal{P}(A)).$$
\end{theorem}

The category of $\widetilde{\mathcal{P}_{\infty}}$-algebras can be equipped with two types of morphisms. The first type of morphisms are morphisms as $\Omega_{op} B_{op} \mathcal{P}$-algebras that commute with all the operations coming from the operad $\Omega_{op} B_{op} \mathcal{P}$, these morphisms will be called strict morphisms of $\widetilde{\mathcal{P}_{\infty}}$-algebras. The second type of morphisms are the morphisms that commute only up to a sequence of coherent homotopies  with the operations coming from the operad $\Omega_{op} B_{op}\mathcal{P}$. These morphism are called $\infty$-morphisms and are defined in the following definitions. The advantage of $\infty$-morphisms is that $\infty$-quasi-isomorphisms are always invertible up to homotopy and that the category of $\widetilde{\mathcal{P}_{\infty}}$-algebras with $\infty$-morphisms is equivalent to the homotopy category of $\mathcal{P}$-algebras. 

\begin{definition}
Let $\iota:B_{op}\P \rightarrow \Omega_{op}B_{op}\P\cong \widetilde{\P_\infty}$ be the canonical twisting morphism.  An $\infty$-morphism $f:A \rightarrow B$ between two $\widetilde{\mathcal{P}_{\infty}}$-algebras $A$ and $A'$ is a morphism $f:B_{\alpha}A \rightarrow B_{\alpha} A'$ of $B_{op} \mathcal{P}$-coalgebras. This is equivalent to a sequence of maps $f_{r_i}:A^{\otimes r} \rightarrow A '$,  where $r_i$ runs over a basis of $B_{op} \mathcal{P}(r)$,  satisfying certain coherence conditions. An $\infty$-morphism $f:A \rightarrow A'$ is called an $\infty$-quasi-isomorphism if the component $f_1:A\rightarrow A'$ is a quasi-isomorphism, where $f_1$ is the morphism corresponding to the operadic unit.
\end{definition}

The main reason we are using $\widetilde{\mathcal{P}_{\infty}}$-algebras in this paper is because of the Homotopy Transfer Theorem.

\begin{theorem}\label{thrmhomotopytransfer}
Suppose that we have a   homotopy retract of a chain complex $(V,d_V)$ to a chain complex $(W,d_W)$, i.e. we have maps

$$\xymatrix{
W \ar@(ul,dl)[]|{h} \ar@/^/[rr]|p
&& V. \ar@/^/[ll]|{i} }$$

  Such that the maps $i$ and $p$ are quasi-isomorphisms of chain complexes and $h$ is a homotopy between $ip$ and the identity, i.e. the maps $i$, $p$ and $h$ satisfy the following identity, 
$$Id_W-ip=d_W h + h d_W.$$
Further suppose that we have a $\mathcal{P}$-structure on $W$, then there exists a $\widetilde{\mathcal{P}_{\infty}}$-structure on $V$  and an $\infty$-morphism $I$ such that $W$ and $V$ are $\infty$-quasi-isomorphic and the map $I$ is an $\infty$-quasi-isomorphism, such that  $I_{(1)}$, the arity one component of $I$, is equal to $i$.
\end{theorem}

Dually we can also define the notion of a coalgebra up to homotopy. This is done completely analogous to the algebra case. 

\begin{definition}
A $\widetilde{\mathcal{C}_{\infty}}$-coalgebra is a coalgebra over the cooperad $B_{op} \Omega_{op} \mathcal{C}$ and an $\infty$-morphism $f:C \rightarrow C'$ of $\widetilde{\mathcal{C}_{\infty}}$-coalgebras is a morphism of $\Omega_{op} \mathcal{C}$-algebras  $f:\Omega_{\pi} \C \rightarrow \Omega_{\pi} C'$, where $\pi:B_{op}\Omega_{op} \C \rightarrow \Omega_{op}\C$ is the canonical twisting morphism. This is equivalent to a sequence of maps $f_{r_i}:C \rightarrow C'^{\otimes r}$, where $r_i$ runs over a basis of $\Omega_{op} \mathcal{C}(r)$, satisfying a sequence of compatibility conditions. 
\end{definition}

Similarly to algebras, we also have a Homotopy Transfer Theorem for $\widetilde{\mathcal{C}_{\infty}}$-coalgebras.

\begin{theorem}
Let $(V,d_V)$ be a retract of $(W,d_W)$ as in Theorem \ref{thrmhomotopytransfer} and suppose that we have $\mathcal{C}$-coalgebra structure on $W$. Then there exists a $\widetilde{\mathcal{C}_{\infty}}$-structure on $V$ such that $V$ and $W$ are $\infty$-quasi-isomorphic and  the morphism $i$ extends to an $\infty$-quasi-isomorphism $I:W \rightarrow V$.
\end{theorem}

This theorem is not explicitly stated or proven in \cite{LV}, but the proof is completely analogous to the algebra case.

Because of the Homotopy Transfer Theorems, it is always possible to equip the homology of a $\mathcal{P}$-algebra $A$ (resp. $\mathcal{C}$-coalgebra $C$) with an $\infty$-structure such that we have a weak equivalence between $A$ and $H_*(A)$ (resp. $C$ and $H_*(C)$). From now on we will always assume that the homology is equipped with the appropriate $\infty$-structure.

\begin{convention}
In the rest of this paper we will assume that whenever we take the homology of a $\mathcal{P}$-algebra $A$ (resp. $\mathcal{C}$-coalgebra $C$) it is equipped with the $\widetilde{\mathcal{P}_{\infty}}$-structure (resp. $\widetilde{\mathcal{C}_{\infty}}$-structure) coming from the Homotopy Transfer Theorem, we therefore have an $\infty$-quasi-isomorphism between $A$ and $H_{*}(A)$ (resp. $C$ and $H_*(C)$).
\end{convention}

\part{Algebraic Hopf invariants}

\section{An $SL_{\infty}$-structure on $Hom_{\K}(C,A)$}\label{seclinfty}

Let $C$ be a $\mathcal{C}$-coalgebra, let $A$ be a $\mathcal{P}$-algebra and let $\tau:\mathcal{C} \rightarrow \mathcal{P}$ be an operadic twisting morphism. In this section we will describe an $SL_{\infty}$-structure on the convolution algebra $Hom_{\K}(C,A)$, such that the Maurer-Cartan elements of this $SL_{\infty}$-algebra are the twisting morphisms relative to the operadic twisting morphism $\tau:\mathcal{C} \rightarrow \mathcal{P}$. We will also explain that the gauge equivalence relation on the Maurer-Cartan elements is equivalent to the homotopy relation on the corresponding morphisms of $\mathcal{P}$-algebras.

To construct an $SL_{\infty}$-structure on $Hom_{\K}(C,A)$, we first prove in Proposition \ref{propconvoperad} that $Hom_{\K}(C,A)$ is an algebra over the convolution operad $Hom_{\K}(\mathcal{C},\mathcal{P})$.  We then construct the $SL_{\infty}$-structure on \\ $Hom_{\K}(C,A)$, by defining a morphism from the $SL_\infty$-operad to $Hom_{\K}(\mathcal{C}, \mathcal{P})$. 

Because the $SL_{\infty}$-operad is defined as the cobar construction on $\mathcal{COCOM}$, specifying a morphism from $SL_{\infty}$ to an operad $\mathcal{Q}$ is the same as specifying a twisting morphism $\mathcal{COCOM} \rightarrow Q$ (see Theorem \ref{thrm6.5.7}). Because of Lemma \ref{lemcocomunit}, there is an isomorphism between $Hom_{\K}(\mathcal{COCOM},\mathcal{Q})$ and $\mathcal{Q}$. A twisting morphism from $\mathcal{COCOM}$ to $\mathcal{P}$ is therefore the same as a Maurer-Cartan element in the pre-Lie algebra associated to $\mathcal{Q}$. 

A twisting morphism $\tau:\mathcal{C} \rightarrow \mathcal{P}$ is therefore equivalent to  a map of operads $SL_{\infty} \rightarrow   Hom_{\K}(\mathcal{C} , \mathcal{P})$ and therefore defines an $SL_{\infty}$-structure on $Hom_{\K}(C,A)$. An explicit description of the $SL_\infty$-structure on $Hom_{\K}(C,A)$ is as follows. Let $f_1,...,f_n \in Hom_\K(C,A)$, then we define the differential of $f_i$ by $d_{Hom_\K(C,A)}(f)=d_A \circ f +(-1)^{\vert f \vert }f \circ d_C$. The higher products  $l_n(f_1,...,f_n)$ are defined by the following sequence of maps
\[
C \xrightarrow{\Delta_C^n} \C \otimes C^{\otimes n} \xrightarrow{\sum_{\sigma \in \Sigma_n}\tau \otimes f_{\sigma(1)} \otimes ... \otimes f_{\sigma(n)}} \P \otimes A^{\otimes n} \xrightarrow{\gamma_A} A,
\]
where $\Delta^n_{\C}$ is the arity $n$ part of the coproduct of $C$ and $\gamma_A$ is the structure map of the $\P$-algebra $A$.

We begin by showing how the convolution algebra is an algebra over the convolution operad.

\begin{proposition}\label{propconvoperad}
Let $(C,\Delta_C)$ be a $\mathcal{C}$-coalgebra and $(A,\mu_A)$ a $\mathcal{P}$-algebra, the convolution algebra $Hom_{\K}(C,A)$ is then an algebra over the convolution operad $Hom_{\K}(\mathcal{C},\mathcal{P})$.
\end{proposition}

\begin{proof}
The algebra structure on $Hom_{\K}(C,A)$ is defined as follows. Let $\gamma \in Hom_{\K}(\mathcal{C},\mathcal{P})(n)$ and $f_1, ... ,f_n \in Hom_{\K}(C,A)$, then we define $\gamma(f_1,..,f_n)$ as 
$$\gamma(f_1,...,f_n):=C \xrightarrow{\Delta_C} \mathcal{C} \circ C \xrightarrow{pr_n} \mathcal{C}(n) \otimes C^{\otimes n} $$
$$\xrightarrow{\gamma \otimes f_1 \otimes ... \otimes f_n} \mathcal{P}(n) \otimes A^{\otimes n} \xrightarrow{\mu_A} A.$$
The map $pr_n$ is here the projection on the arity $n$ part of the composition product. It is a straightforward consequence of the definition of the convolution operad that this defines an $Hom_{\K}(\mathcal{C},\mathcal{P})$ algebra structure on the dg vector space $Hom_{\K}(C,A)$.
\end{proof}

The next step is to show that the set of morphisms from the $SL_{\infty}$-operad to an operad $\mathcal{Q}$ is  isomorphic to the set of Maurer-Cartan elements in $\mathcal{Q}$.



The main theorem of this section is the following.

\begin{theorem}\label{thrmlinfty}
Let $\tau:\mathcal{C} \rightarrow \mathcal{P}$ be an operadic twisting morphism from a  cooperad $\mathcal{C}$ to an operad $\mathcal{P}$. Let $C$ be a $\mathcal{C}$-coalgebra and let  $A$ be a $\mathcal{P}$-algebra, then the $SL_{\infty}$-structure on $Hom_{\K}(C,A)$ described above is natural in both $C$  and $A$ and has the following properties. 
\begin{enumerate} 
  \item The Maurer-Cartan elements with respect to this $SL_{\infty}$-structure are the twisting morphisms relative to $\tau$ in $Hom_{\K}(C,A)$.
  \item Let $\hat{f}$ and $\hat{g}$ be two Maurer-Cartan elements in $Hom_{\K}(C,A)$ and let $f:\Omega_{\tau} C \rightarrow A$ and $g:\Omega_{\tau} C \rightarrow A$ be the corresponding $\mathcal{P}$-algebra maps, then $f$ and $g$ are homotopic in the model category of $\mathcal{P}$-algebras if and only if the Maurer-Cartan elements $\hat{f}$ and $\hat{g}$ are gauge equivalent. 
\end{enumerate}
\end{theorem}

\begin{proof}
According to \cite{LV}, the set of twisting morphisms is given by all $\phi \in Hom_\K(C,A)_0$, such that $\phi$ satisfies the Maurer-Cartan equation,
$$\partial(\phi)+\star_{\tau}(\phi)=0,$$
where $\star_{\tau}$ is the operator given by
$$\star_{\tau}(\phi):C \xrightarrow{\Delta_C} \mathcal{C} \circ C \xrightarrow{\tau \circ \phi} \mathcal{P} \circ A \xrightarrow{\mu_A} A.$$
So we have to show that the Maurer-Cartan equation from \cite{LV} is the same as the Maurer-Cartan equation in the $SL_{\infty}$-algebra $Hom_{\K}(C,A)$. Therefore we will first make the operations $l_n$ in the $SL_{\infty}$-structure explicit. This is done in a similar way as for the Lie algebra case  in section 11.1.2 of \cite{LV}. The operation $l_n(x_1,...,x_n)$ is defined as the image of $\mu_n$, the arity $n$ element of $\mathcal{COCOM}$, under the twisting morphism $\tau$. When worked out explicitly it is given  by the composite
$$l_n(x_1,...,x_n):C \xrightarrow{\Delta_C} \mathcal{C}(C) \twoheadrightarrow \left( \mathcal{C}(n) \otimes C^{\otimes n} \right)^{\Sigma_n} \rightarrow \mathcal{C}(n)\otimes C^{\otimes C}$$
$$\xrightarrow{\sum_{\sigma \in \Sigma_n} (-1)^{\epsilon} \tau\otimes x_{\sigma (1)} \otimes ... \otimes x_{\sigma (n)}} \mathcal{P}(n) \otimes A^{\otimes n} \rightarrow A.$$
Where the sign $\epsilon$ is coming from the Koszul sign rule. From this it follows that $\star_{\tau}(\phi)$ is equal to $\sum_{n \geq 2} \frac{1}{n!} l_n(\phi,...,\phi)$, therefore the Maurer-Cartan equation defining twisting morphisms is equal to the Maurer-Cartan equation coming from the $SL_{\infty}$-structure. 


To show that the $SL_{\infty}$-structure is natural in $A$, we observe that a morphism $f:A \rightarrow B$ of $\mathcal{P}$-algebras induces a map between convolution algebras $\tilde{f}:Hom_{\K}(C,A) \rightarrow Hom_{\K}(C,B)$ by composition with $f$. Since $f$ is an algebra morphism it commutes with the $\mathcal{P}$ multiplication maps $\mu_A:\mathcal{P} \circ A \rightarrow A$ and $\mu_B:\mathcal{P}  \circ B \rightarrow B$. Therefore it induces a morphism between the corresponding $SL_{\infty}$-algebras, which proves the naturality in $A$, the naturality in $C$ is shown in an analogous manner. 

We prove the implicaition that gauge equivalence implies homotopy equivalence of the second part of the theorem as follows. Recall from \cite{Val1}, that a path object for a $\P$-algebra $A$ is given by $A \otimes \Omega_1$ and that two maps $f:\Omega_\tau C \rightarrow A$ and $g:\Omega_\tau C \rightarrow A$ are homotopic if there exists a map $H:\Omega_\tau C \rightarrow A\otimes \Omega_1$, such that $H\vert_{t=0}=f$ and $H\vert_{t=1}=g$. Suppose that we have a gauge equivalence $\hat{H}$ between $\hat{f}$ and $\hat{g}$, this is a Maurer-Cartan element $\hat{H} \in Hom_\K(C,A) \otimes \Omega_1$. Note that we have an inclusion of $Hom_{\K}(C,A)\otimes \Omega_1 \subset Hom_{\K}(C,A \otimes \Omega_1)$. So every gauge $\hat{H}$ between $\hat{f}$ and $\hat{g}$ defines a homotopy between $f$ and $g$. 

The converse is harder to prove since $\Omega_1$ is infinite-dimensional. Because $\Omega_1$ is infinite-dimensional we do not have an isomorphism  between $Hom_{\K}(C,A)\otimes \Omega_1$ and $Hom_{\K}(C,A \otimes \Omega_1)$, we therefore need to use some different techniques and this is done in Theorem 2.4 of \cite{RNW2}.

\end{proof}

Several versions of this theorem are already known in the literature, in \cite{LV} Loday and Vallette prove a similar result for binary quadratic operads and cooperads in Proposition 11.1.1 and Corollary 11.1.2 and in \cite{DHR1} Dolgushev, Hoffnung and Rogers show that there exists an $SL_{\infty}$-structure on the convolution algebra when $\mathcal{C}=B_{op} \mathcal{P}$. This theorem generalizes both results and gives a shorter and more conceptual proof than in \cite{DHR1}.

\section{Example: The classical Hopf invariant}

In this section we will describe a chain level version of the classical Hopf invariant and explain how we will generalize it to the algebraic setting. This section should be seen as a motivation for our generalizations in the next sections. Most of this section comes from  Example 1.7 in Section 1 of \cite{SW2}. In this section, and only in this section, we use the cochains with coefficients in $\Z$.

The classical Hopf invariant is an invariant of maps $f:S^3 \rightarrow S^2$ and is constructed using the associative bar construction. The goal is to construct a pairing $\eta:H_*(S^3) \times H_* (B_{\tau} C^*(S^2)) \rightarrow \Z$, which is an invariant of the homotopy classes of maps from $S^3$ to $S^2$. To do this denote by $C^*(S^2)$  the singular cochains on $S^2$ and let $\omega \in C^2(S^2)$ be  a cocycle that represents the generator of the cohomology of $S^2$, such that $\omega \cup \omega=0$. Since the singular cochains are an associative algebra we can take the associative bar construction $B_{\tau} (C^*(S^2))$ with respect to the twisting morphism $\tau:Ass^{\vee}_{1} \rightarrow Ass_{0}$ between the coassociative operad with a coproduct of degree $1$ and the associative operad with a product of degree $0$. Since $\omega$ is a cocycle and $\omega^2=0$, the  element $\omega \otimes \omega$ is then also a cycle in the bar construction. 

The next step is to  pull the form $\omega \otimes \omega$ back to $f^* \omega \otimes f^* \omega \in B_{\tau} C^*(S^3)$. The cocycle $f^*\omega$ is exact since $H^2(S^3)=0$, therefore   there exists a form $d^{-1} f^* \omega \in C^1(S^3)$. It is straightforward to see that $d^{-1}f^*\omega \otimes f^* \omega$  is a coboundary cobounding $f^* \omega \otimes f^* \omega$ and $d^{-1} f^* \omega \cup f^* \omega$, where $\cup$ is the chain level version of the cup product on $C^*(S^3)$. The cocycle $f^*\omega \otimes f^* \omega$ is therefore homologous to $d^{-1} f^* \omega \cup f^* \omega$.

This version of the classical Hopf invariant of a map $f:S^3 \rightarrow S^2$ is now defined as
$$\int_{\alpha}(d^{-1}f^* \omega \cup f \omega) \in \Z.$$
Which is the evaluation of the $3$-form $d^{-1}f^* \omega \cup f^* \omega$ on the fundamental class $\alpha$ of $S^3$. It can be shown that this construction is independent of choices and defines an invariant of the map $f$, which is equal to Hopf's classical definition of the Hopf invariant. 

So what we have done in this example is specifying a pairing 
$$\eta:H_*(S^3) \otimes H^*(B_{\tau} C^*(S^2)) \rightarrow \Z.$$
This is equivalent to specifying a linear map  $H_*(S^3) \rightarrow H_*(\Omega_{\tau} C_*(S^2))$. The goal of our generalization of the Hopf invariant is therefore to associate to each map of spaces $f:X \rightarrow Y_{\Q}$ a linear map $mc_{\infty}:\tilde{H}_*(X) \rightarrow \pi_*(Y_{\Q})$ which is an invariant of the homotopy class of the map $f$.


\section{Algebraic Hopf invariants}

Let $C$ and $D$ be coalgebras over a cooperad $\mathcal{C}$ and let $\iota:\mathcal{C} \rightarrow \Omega_{op} \mathcal{C}$ be the canonical twisting morphism from $\mathcal{C}$ to its cobar construction $\Omega_{op}\C$. The goal of this section is to construct a map $mc_{\infty}:Hom_{\mathcal{C}-coalg}(C,D) \rightarrow \mathcal{MC}(H_*(C),H_*(\Omega_{\iota}D))$, from the set of coalgebra maps from $C$ to $D$, to a certain moduli space of Maurer-Cartan elements associated to $C$ and $D$. The map $mc_{\infty}$ has the property that it is a complete invariant of homotopy classes of maps, i.e. two maps $f$ and $g$ are homotopic if and only if $mc_{\infty}(f)=mc_{\infty}(g)$. 

The map $mc_{\infty}$ will be constructed in two steps. The first step is to define a map 
\[
mc:Hom_{\mathcal{C}-coalg}(C,D) \rightarrow Hom_{\K}(H_*(C),H_*( \Omega_{\iota} D)),
\]
which assigns to each coalgebra map a Maurer-Cartan element in the convolution algebra between the homology of $C$ and the homology of the cobar construction of $D$. The cobar construction here is taken with respect to the canonical twisting morphism $\iota:\mathcal{C} \rightarrow \Omega_{op} \mathcal{C}$. The map $mc$ is not homotopy invariant yet, to make it homotopy invariant we have to pass to the moduli space of Maurer-Cartan elements. The second step is therefore to compose this map with the quotient map onto the moduli space of Maurer-Cartan elements.

\begin{remark}
 To talk about gauge equivalence on the space of Maurer-Cartan elements in \\ $Hom_{\K}(H_*(C),H_*(\Omega_{\iota} D))$, it is necessary to have an $SL_{\infty}$-structure, we will define this $SL_{\infty}$-structure in Section \ref{secmappingmodel}.
\end{remark}

The map $mc:Hom_{\mathcal{C}-coalg}(C,D) \rightarrow Hom_{\K}(H_*(C),H_*(\Omega_{\iota} D))$ is not canonical and to construct it we first need to make a couple of choices. First, we pick an $\infty$-quasi-isomorphism $i:H_*(C)\rightarrow C$, where we assume that $H_*(C)$ has a $\widetilde{\mathcal{C}_{\infty}}$-structure coming from the Homotopy Transfer Theorem. Then we pick a strict morphism of $\Omega_{op} \mathcal{C}$-algebras $p:\Omega_{\iota} D \rightarrow H_*(\Omega_{\iota} D)$, where we again equip $H_*(\Omega_{\iota} D)$ with the transferred $\Omega_{op} \mathcal{C}$-structure coming from the Homotopy Transfer Theorem. The morphism $p$ can be chosen in such a way that it is a strict morphism of $\Omega_{\iota} \mathcal{C}$-algebras, and not just an $\infty$-morphism, because of the following proposition.  

\begin{proposition}\label{propeverythingisformal}
Let $p:\Omega_{\iota}D \rightarrow H_*(\Omega_{\iota} D)$ be a choice of projection of chain complexes of $\Omega_{\iota}D$ onto its homology. Then  there exists a strict morphism of $\Omega_{op} \mathcal{C}$-algebras $p':\Omega_{\iota} D \rightarrow H_*(\Omega_{\iota} D)$ which has the linear map $p\vert_D:D \rightarrow H_*(\Omega_{\iota}D)$, defined by the restriction of $p$ to $D$, as its Maurer-Cartan element. 
\end{proposition}


\begin{proof}
To prove the proposition we will first define an $\infty$-morphism $P:\Omega_{\iota}D \rightarrow H_*(\Omega_{\iota}D)$ and then rectify this to a strict morphism with Maurer-Cartan element $p$. To define the $\infty$-morphism $P: \Omega_{\iota}D \rightarrow H_*(\Omega_{\iota } D)$ we will first pick maps $j:H_*(\Omega_{\iota}D) \rightarrow \Omega_{\iota}D$ and $H:\Omega_{\iota}D \rightarrow \Omega_{\iota} D$ such that $p$, $j$ and $H$ form a contraction as in Theorem \ref{thrmhomotopytransfer}. Using Theorem \ref{thrmhomotopytransfer} we now get a $\Omega_{op}\mathcal{C}$-algebra structure on $H_*(\Omega_{\iota}D)$ and an $\infty$-morphism $p_{\infty}:\Omega_{\iota}D\rightarrow H_*( \Omega_{\iota}D)$. 

 To  construct the strict morphism $p':\Omega_{\iota}D \rightarrow H_*(\Omega_{\iota} D)$ we will first rectify the $\infty$-morphism $p_{\infty}$.   According to Theorem 11.4.13  from \cite{LV}, we can rectify $p_{\infty}$ as follows,
 $$\Omega_{\iota}D \xleftarrow{\epsilon_{\Omega_{\iota}D}} \Omega_{\iota} B_{\iota} \Omega_{\iota} D \xrightarrow{\Omega_{\iota} B_{\iota} p_{\infty}} \Omega_{\iota} B_{\iota} H_*(\Omega_{\iota} D) \xrightarrow{\epsilon_{H_*(\Omega_{\iota} D)}} H_*(\Omega_{\iota} D).$$
 The map $\epsilon_A:\Omega_{\iota} B_{\iota} A \rightarrow A$ here is the counit of the bar-cobar adjunction and is given by the projection of $\Omega_{\iota} B_{\iota} A$ onto $A$. We define a strict morphism of $\Omega_{op} \mathcal{C}$-algebras from $\Omega_{\iota}D$ to $H_*(\Omega_{\iota} D)$ by taking an inverse up to homotopy  of the map $\epsilon_{\Omega_{\iota} D}$, which is possible since $\Omega_{\iota}D$ is cofibrant. An explicit homotopy inverse is given by $\Omega_{\iota} \eta_{D}:\Omega_{\iota} D \rightarrow \Omega_{\iota} B_{\iota} \Omega_{\iota} D$, where $\eta_{D}:D \rightarrow B_{\iota} \Omega_{\iota} D$ is the unit of the bar-cobar adjunction and is given by the inclusion of $D$ into $B_{\iota} \Omega_{\iota} D$. So we get a sequence of maps 
 $$\Omega_{\iota}D \xrightarrow{\Omega_{\iota} \eta_{D}} \Omega_{\iota} B_{\iota} \Omega_{\iota} D \xrightarrow{\Omega_{\iota} B_{\iota} p_{\infty}} \Omega_{\iota} B_{\iota} H_*(\Omega_{\iota} D) \xrightarrow{\epsilon_{H_*(\Omega_{\iota} D)}} H_*(\Omega_{\iota} D) $$
 whose composition is a strict morphism of $\Omega_{op} \mathcal{C}$-algebras. The Maurer-Cartan element corresponding to this map is given by the image of $D$ under the composite $\epsilon_{H_*(\Omega_{\iota} D)} \circ \Omega_{\iota} B_{\iota} p_{\infty} \circ \Omega_{\iota} \eta_D$, so we have to chase the image of  $D$ through all these maps. The image of $D$ under $\Omega_{\iota} \eta_D$ is given by the subspace $D \subset \Omega_{\iota} B_{\iota} \Omega_{\iota} D$, the image of $D$ under the map   $\Omega_{\iota} B_{\iota} p_{\infty}$ is contained in the subspace $B_{\iota} H_*(\Omega_{\iota} D)$. Since the map $\epsilon_{H_*(\Omega_{\iota} D)}$ is the projection onto $H_*(\Omega_{\iota} D)$, therefore the image of the composition is given by the image of $D$ onto $H_*(\Omega_{\iota} D)$ which is equal to the map $p$. Therefore the Maurer-Cartan element corresponding to the map  $p'$ is equal to $p$, which proves the proposition.
\end{proof}

\begin{remark}
 Note that in the proof of Proposition \ref{propeverythingisformal} by defining the strict morphism $p':\Omega_{\iota}D \rightarrow H_*(\Omega_{\iota}D)$ we also changed the $B_{op}\Omega_{op}\mathcal{C}$-structure on $H_*(\Omega_{\iota}D)$. If we would first fix an $B_{op}\Omega_{op}\mathcal{C}$-structure on $H_*(\Omega_{\iota}D)$ and then try to construct $p'$, Proposition \ref{propeverythingisformal} would not be true.
\end{remark}

Now that we have fixed the maps $i$ and $p$, we define the map 
$$mc:Hom_{\mathcal{C}-coalg}(C,D) \rightarrow Hom_{\K}(H_*(C),H_*(\Omega_{\iota}D))$$
as follows
$$mc(f)=p \circ \Omega_{\iota} f \circ \Omega_{\iota} i. $$
So we first take the cobar construction of $f$ and then precompose it with $\Omega_{\iota} i$ and compose it with $p$. This map is not yet homotopy invariant, but by Theorem \ref{thrmlinfty}, homotopic maps will have gauge equivalent values. Therefore we define $mc_{\infty}:Hom_{\mathcal{C}-coalg}(C,D)\rightarrow \mathcal{MC}(H_*(C),H_*(\Omega_{\iota} D))$ as the map which sends $f$ to the equivalence class of $mc(f)$ in the moduli space of Maurer-Cartan elements.

\begin{definition}
The algebraic Hopf invariant $mc_{\infty}(f)$ of a map $f:C \rightarrow D$ is defined as the image of the map $mc_{\infty}:Hom_{\mathcal{C}-coalg}(C,D)\rightarrow \mathcal{MC}(H_*(C),H_*(\Omega_{\iota} D))$, where $mc_{\infty}$ is the map that sends $f$ to the equivalence class of $mc(f)$ in the moduli space of Maurer-Cartan elements.
\end{definition}

We will spend the rest of this section to show that $mc_{\infty}$ is  a well defined invariant of the set of homotopy classes of maps. In the next section we will show that it is a complete invariant as well.

\begin{proposition}\label{prophopfinv}
 For every choice of maps $i:H_*(C) \rightarrow C$ and $p:\Omega_{\iota}D \rightarrow H_*(\Omega_{\iota} D)$, the algebraic Hopf invariant is  an invariant of the homotopy class of the map $f$.
\end{proposition}

Before we prove the proposition we will first prove the following lemma. This lemma should be well known, but because we could not find a reference we decided to include its proof for the sake of completeness.

\begin{lemma}\label{lem3.15}
 Let $f,g:C \rightarrow D$ be left homotopic maps in a model category $\mathbf{C}$, assume that the objects $C$ and $D$ are cofibrant and let $F:\mathbf{C} \rightarrow \mathbf{D}$ be a left Quillen functor from $\mathbf{C}$ to a model category $\mathbf{D}$. Then $F$ preserves left homotopies, i.e. $F(f)$ is left homotopic to $F(g)$.
\end{lemma}

\begin{proof}
 Let $Cyl(C)$ be a good cylinder object for $C$, which gives a homotopy between $f$ and $g$. This is  an object $Cyl(C)$, such that the first map is a cofibration and the second map a weak equivalence $C\sqcup C  \hookrightarrow Cyl(C) \xrightarrow{~} C$, where $C \sqcup C$ is the coproduct. Since $f$ and $g$ are homotopic, this means that there is a map $H:Cyl(C) \rightarrow D$, such that $H$ restricted to the first factor is  $f$ and the restriction of $H$  to the second factor is $g$. So what we would like to prove is that $F(Cyl(C))$ is a cylinder object for $F(C)$, because then the map $F(H):F(Cyl(C)) \rightarrow F(D)$ is a homotopy between $F(f)$ and $F(g)$. 
 
 To prove that $F(Cyl(C))$ is a cylinder object for $F(C)$ we first observe that since $F$ is a left Quillen functor it preserves weak equivalences between cofibrant objects and because it is a left adjoint it preserves coproducts. Since we assumed that $C$ is cofibrant  this implies that $C\sqcup C$ and $Cyl(C)$ are also cofibrant. So if we apply the functor $F$ then we get $F(C \sqcup  C) \hookrightarrow F(Cyl(C)) \xrightarrow{~} F(C)$ which is equal to $F(C) \sqcup F(C) \hookrightarrow F(Cyl(C)) \xrightarrow{~} F(C)$. Therefore $F(Cyl(C))$ is a cylinder object for $F(C)$ and $F(H)$ a homotopy between $F(f)$ and $F(g)$. 
\end{proof}

The homotopy invariance of the Hopf invariants follows from this lemma if we apply it to the cobar construction.

\begin{proof}[Proof of Proposition \ref{prophopfinv}]
The homotopy invariance of the algebraic Hopf invariants follows from Lemma \ref{lem3.15} and Theorem \ref{thrmlinfty}. More precisely if $f:C \rightarrow D$ and $g:C \rightarrow D$ are homotopic maps then $\Omega_{\iota}f$ and $\Omega_{\iota}g$ are homotopic because of Lemma \ref{lem3.15}. The lemma applies since $\Omega_{\iota}$ is a left Quillen functor and every coalgebra is cofibrant. Since $\Omega_{\iota}f$ and $\Omega_{\iota}g$ are homotopic, so are $p \circ \Omega_{\iota} f \circ i$ and $p \circ \Omega_{\iota} g \circ g$. Then because of Theorem \ref{thrmlinfty} the maps $p \circ \Omega_{\iota}f \circ i$ and $p \circ \Omega_{\iota}g \circ i$ are homotopic if and only if the corresponding Maurer-Cartan elements are gauge equivalent. Therefore the algebraic Hopf invariants are an invariant of the homotopy class of the map $f$. 

\end{proof}

\begin{remark}
The fact that we need to make a choice for the maps $i$ and $p$ is already visible in the case of the classical Hopf invariant. The map $i$ can be seen as the choice of an orientation of the fundamental class of $S^{4n-1}$.  If we do not pick an orientation the  classical Hopf invariant is only well defined up to a sign. This choice of  orientation is exactly what the map $i$ does in our construction, it fixes a set of representatives for the homology for which we compute the corresponding Maurer-Cartan element.
 
 Since we need to make a choice for $i$ and $p$, we will assume for simplicity that in the rest of this paper the maps $i$ and $p$ are fixed and we therefore omit them from the notation.
\end{remark}

\section{Completeness of the algebraic Hopf invariants}\label{seccompleteness}

In this section we will prove the main theorem about the algebraic Hopf invariants, which is that they form a complete invariant of the set of homotopy classes of maps $[C,D]$,  between two $\mathcal{C}$-coalgebras $C$ and $D$.

\begin{theorem}\label{thrmcompletehopfinvariants}
Let $f:C \rightarrow D$ and $g:C \rightarrow D$ be two maps of $\mathcal{C} $ coalgebras, then $f$ and $g$ are homotopic in the model category of $\mathcal{C}$-coalgebras if and only if they have the same algebraic Hopf invariant. 
\end{theorem}

\begin{proof}
To prove the theorem we want to show that the algebraic Hopf invariant map $mc_{\infty}$ induces a bijection between the set of homotopy classes of maps $[C,D]$ and the moduli space of Maurer-Cartan elements $\mathcal{MC}(H_*(C),H_*(\Omega_{\iota}D))$. Since two maps are homotopic if and only if they belong to the same homotopy class of maps, this bijection implies that two maps are homotopic if and only if they have the same algebraic Hopf invariant. 

To show that there is a bijection between $[C,D]$ and  $\mathcal{MC}(H_*(C),H_*(\Omega_{\iota}D))$ we will show that we have the  following sequence of bijections.
$$[C,D] \cong [\Omega_{\iota}C,\Omega_{\iota} D] \cong  [\Omega_{\iota} C, H_*(\Omega_{\iota} D)] \cong[\Omega_{\iota} H_*(C), H_*(\Omega_{\iota} D)] \cong \mathcal{MC}(H_*(C),H_*(\Omega_{\iota}D)].$$

The first bijection between $[C,D]$ and $[\Omega_{\iota}C,\Omega_{\iota}D]$ follows from the fact that $\Omega_{\iota}$ is the left Quillen functor in a Quillen equivalence and therefore induces a bijection on the level of homotopy categories.

The second bijection between  $[\Omega_{\iota}C,\Omega_{\iota} D]$ and $ [\Omega_{\iota} C, H_*(\Omega_{\iota} D)]$ is given by the map induced by composition with the map $p:\Omega_{\iota} D\rightarrow H_*(\Omega_{\iota} D)$ from Lemma \ref{propeverythingisformal}. Since the map $p$ is a surjective quasi-isomorphism, it is an acyclic fibration, Lemma \ref{lemacyclicfibrations} then implies that the second map is indeed a bijection.

To show that we have the third bijection between $[\Omega_{\iota} C, H_*(\Omega_{\iota} D)]$ and $[\Omega_{\iota} H_*(C), H_*(\Omega_{\iota} D)]$, we want to use Lemma \ref{lemkbrown}. We apply this lemma to the functor $[-,H_*(\Omega_{\iota}D)]:\Omega_{op}\mathcal{C} -algebras \rightarrow Sets$ which sends an algebra $A$ to the set of homotopy classes of maps between $A$ and $H_*(\Omega_{\iota}D)$. The category of sets is here equipped with the trivial model structure in which the weak equivalences are given by bijections. This functor carries acyclic cofibrations to weak equivalences because of a dual version of Lemma \ref{lemacyclicfibrations}, therefore K. Brown's Lemma applies. Since both $\Omega_{\iota} C$ and $ \Omega_{\iota} H_*(C)$ are cofibrant and the map $\Omega_{\iota} i : \Omega_{\iota} H_*(C) \rightarrow \Omega_{\iota} C$ is a weak equivalence, the map $[\Omega_{\iota} C, H_*(\Omega_{\iota} D)] \rightarrow [\Omega_{\iota} H_*(C), H_*(\Omega_{\iota} D)]$ is a weak equivalence in the category of sets and therefore a bijection.

The last bijection between $[\Omega_{\iota} H_*(C), H_*(\Omega_{\iota} D)]$ and   $\mathcal{MC}(H_*(C),H_*(\Omega_{\iota}D))$ follows from Theorem \ref{thrmlinfty} and the fact that the set of twisting morphisms is represented by the cobar construction. More precisely, the set of algebra morphisms $Hom_{\Omega_{op}\mathcal{C}-alg}(\Omega_{\iota} H_*(C), H_*(\Omega_{\iota} D))$ is in bijection with the Maurer-Cartan elements in the convolution $SL_{\infty}$-algebra $Hom_{\K}(H_*(C), H_*(\Omega_{\iota} D))$. Because of Theorem \ref{thrmlinfty} the homotopy equivalence relation on $Hom_{\Omega_{op}\mathcal{C}-alg}(\Omega_{\iota} H_*(C), H_*(\Omega_{\iota} D))$ is equivalent to the gauge equivalence relation on  $Hom_{\K}(H_*(C), H_*(\Omega_{\iota} D))$. Therefore there is a bijection between $[\Omega_{\iota} H_*(C), H_*(\Omega_{\iota} D)]$ and  $\mathcal{MC}((H_*(C), H_*(\Omega_{\iota} D))$.

Putting all these bijections together we get a bijection between  $\mathcal{MC}(H_*(C),H_*(\Omega_{\iota} D))$ and $[C,D]$. The maps $f$ and $g$ are therefore homotopic if and only if $mc_{\infty}(f)=mc_{\infty}(g)$, which proves the theorem.
 
\end{proof}

\section{Models for mapping spaces}\label{secmappingmodel}

In this section we use to the $SL_\infty$-structure from Theorem \ref{thrmlinfty} to equip the chain complex \\ $Hom_{\K}(H_*(C),H_*( \Omega_{\iota} D))$ with the structure of an $SL_{\infty}$-algebra. This is necessary because we need a notion of gauge equivalence between the Maurer-Cartan elements in $Hom_{\K}(H_*(C),H_*( \Omega_{\iota} D))$ to pass to the moduli space of Maurer-Cartan elements. We also explain that when we specialize ourselves to the cooperad $C_{\infty}$, we can use this $SL_{\infty}$-algebra to construct rational models for mapping spaces.

To equip $Hom_{\K}(H_*(C),H_*( \Omega_{\iota} D))$ with an $SL_\infty$-structure, we need a twisting morphism from $\widetilde{\C_\infty}= B_{op} \Omega_{op} \C $ to $ \Omega_{op}\C$. A canonical choice for this twisting morphism is the canonical twisting morphism $\pi:B_{op} \Omega_{op} \C \rightarrow \Omega_{op}\C$. Using this twisting morphism we can equip $Hom_{\K}(H_*(C),H_*( \Omega_{\iota} D))$ with the correct $SL_\infty$-structure, this also allows us to construct the moduli space of Maurer-Cartan elements.

With some small modifications we can also use this $SL_\infty$-structure to construct rational models for mapping spaces. Let $X$ and $Y$ be simply-connected spaces of finite $\Q$-type, let $A$ be a non-unital CDGA model for $X$ and let $L$ be an $SL_{\infty}$-model for $Y$. In \cite{Berg1} Theorem 6.3, Berglund constructed an $SL_{\infty}$-model for the based mapping space $Map_*(X,Y)$. This model is given by $A \hat{\otimes} L$, the completed tensor product of $A$ and $L$ with a certain $SL_\infty$-structure. 

Although this construction works in many interesting cases, it has the obvious disadvantage that it assumes that the source $A$ has to be a CDGA instead of a more general $C_{\infty}$-algebra. In this section we state some of the results of \cite{RNW1}, which show that the $SL_{\infty}$-structure from Section \ref{seclinfty} can be used to generalize Berglund's Theorem. This generalization extends his result by constructing a rational model for the mapping space from a $C_{\infty}$-coalgebra and an $SL_{\infty}$-algebra. As a corollary we give an alternative proof for  Theorem 3.2 of \cite{BG1}. In this theorem Buijs and Guti\'errez show that $\tilde{H}^*(X)\otimes \pi_*(Y)$ can be equipped with an $SL_{\infty}$-structure, such that it becomes a rational model for the mapping space. From now on it will be necessary to have some restrictions on the $SL_{\infty}$-algebras we are considering. This is done in the following convention.

\begin{convention}
From now on we will assume that all $SL_{\infty}$-algebras are degree-wise nilpotent (see Definition \ref{defnilpotence}). Under some mild assumptions on $X$ and $Y$, the models for the mapping space $Map_*(X,Y)$ will always be degree-wise nilpotent. An example of such assumptions would be to assume that is $X$ a finite 1-reduced CW-complex and $Y$ a simply-connected rational space of finite $\Q$-type. Therefore we will from now on tacitly assume that all the $SL_{\infty}$-algebras we encounter are degree-wise nilpotent.
\end{convention}


%
%

%
%

%
 We will combine Berglund's results and the $SL_{\infty}$-structure from section \ref{seclinfty}, to construct a model for the mapping space as follows. Let $C$ be a finite-type $C_{\infty}$-coalgebra model for a simply-connected  CW-complex $X$ of finite $\Q$-type and let $L$ be a rational $SL_{\infty}$-model of finite type for a space $Y$ of finite $\Q$-type, denote by $Y_{\Q}$ the rationalization of $Y$. We would like to apply Theorem \ref{thrmlinfty} to the convolution algebra $Hom_{\K}(C,A)$, but to apply this theorem we first need a twisting morphism $\tau:C_{\infty} \rightarrow SL_{\infty}$.  In Theorem 6.5.10 of \cite{LV}, it is shown that the set of twisting morphisms in $Hom_{\K}(C_{\infty},SL_{\infty})$ is represented by the set of cooperad maps $Hom_{coop}(C_{\infty},B_{op} SL_{\infty})$. Therefore to construct the twisting morphism $\tau$, we pick  a quasi-isomorphism $\phi:C_{\infty} \rightarrow B_{op} \Omega_{op} \mathcal{COCOM}$ from the bar-cobar resolution of the cocommutative cooperad $\mathcal{COCOM}$ to $C_{\infty}$. Since $C_{\infty}$ and $B_{op} \Omega_{op} \mathcal{COCOM}$ are both fibrant and cofibrant this is always possible. The twisting morphism $\tau$ is then defined as the composition of $\phi$ with the universal twisting morphism $\pi:B_{op} SL_{\infty} \rightarrow SL_{\infty}$, which is given by the projection onto $SL_{\infty}$. The twisting morphism $\tau$ is then defined by  $\tau := \pi \circ \phi$. The morphism $\phi$ also allows us to view every $C_{\infty}$-coalgebra as a $B_{op} \Omega_{op} \mathcal{COCOM}$-coalgebra. Because this does not change the underlying chain complex and therefore not the homology, it will be convenient to view every $C_{\infty}$-coalgebra as a $B_{op}\Omega_{op} \mathcal{COCOM}$-coalgebra. Therefore we will from now view every $C_{\infty}$-coalgebra as an $B_{op} \Omega_{op} \mathcal{COCOM}$-coalgebra, unless stated otherwise, this has the advantage that we can work with the twisting morphism $\pi$ which is easier than $\tau$. 
 
\begin{remark}
There is a natural choice for the quasi-isomorphism $\phi$, which is given by the bar construction on the morphism $\mathcal{LIE}^{\vee}\rightarrow SL_{\infty}^{\vee}$. For our results it does not matter which quasi-isomorphism we are using, we will therefore not pick a specific twisting morphism.
\end{remark}

Before we define the model for the mapping space, we first need one more restriction on the $SL_\infty$-algebras we are working with.

\begin{definition}
An $SL_{\infty}$-algebra $L$ is called locally finite if the filtration quotients $L / \Gamma_nL$ are finite dimensional, where $\Gamma_nL$ is the lower central series of $L$ (see Definition \ref{defnilpotence}).
\end{definition}

\begin{theorem}[\cite{RNW1}, Theorem  9.19]\label{thrmmappingmodel}
  Let $X$  be a simply-connected CW-complex of finite $\Q$-type and let $Y$ be a simply-connected space of finite $\Q$-type and $Y_{\Q}$ its $\Q$-localization. Let $C$ be a $C_{\infty}$-coalgebra model of finite type for $X$ and let $L$ be  simply-connected locally finite $SL_{\infty}$-model for $Y_{\Q}$. The convolution $SL_{\infty}$-algebra $Hom_{\K}(C,L)$ is a model for the mapping space $Map_*(X,Y_{\Q})$, i.e. there is a homotopy equivalence between 
 $$Map_*(X,Y_{\Q}) \simeq MC_{\bullet}(Hom_{\K}(C,L)).$$
\end{theorem}

The idea of the proof is as follows. We would like to compare the $SL_{\infty}$-algebra $Hom_{\K}(C,L)$ to Berglund's model of the mapping space. To do this we first dualize $C$ to $C^\vee$, as is shown in Theorem 9.14 of \cite{RNW1}, the dual of a coalgebra model is an algebra model. Unfortunately Berglund's model does not apply yet, since $C^{\vee}$ is not commutative, but only commutative up to homotopy. We therefore need to replace $C^\vee$ be a commutative algebra, this is done by using the rectification from Section 11.4.3 of \cite{LV}. More precisely, we can find a commutative algebra $R(C^\vee)$  and an $\infty$-quasi-isomorphism $\beta:R(C^\vee) \rightarrow C^\vee$. According to Berglund's Theorem, the space $R(C^\vee) \otimes L$ is now an  $SL_\infty$-model for $Map_*(X,Y_\Q)$. It is shown in Theorem 8.10 of \cite{RNW1}, that under some conditions, $\infty$-quasi-morphisms induce homotopy equivalences between Maurer-Cartan simplicial sets. This implies that we have a homotopy equivalence between $Hom_\K(C,L)$ and $Map_*(X,Y_\Q)$. For more details see the proof of Theorem 9.19 of \cite{RNW1}.

As a corollary of Theorem \ref{thrmmappingmodel} we find a new proof for  Theorem 3.2 in the paper \cite{BG1} by Buijs and Guti\'errez. This theorem states that  $Hom_{\K}(\tilde{H}_*(X),\pi_*(Y))$ can be equipped with an $SL_{\infty}$-structure, such that it becomes a rational model for the mapping space $Map_*(X,Y_{\Q})$.

\begin{corollary}
 Let $X$ be a finite 1-reduced CW-complex and $Y_{\Q}$ a simply-connected rational space of finite $\Q$-type, then the space  $Hom_{\K}(\tilde{H}_*(X),\pi_*(Y))$ can be equipped with an $SL_{\infty}$-structure such that 
 $$MC_{\bullet}(Hom_{\K}(\tilde{H}_*(X),\pi_*(Y)))\simeq Map_*(X,Y_{\Q}).$$
\end{corollary}

\begin{proof}
 To prove the corollary we define $H$ as the reduced homology of $X$ with a $C_{\infty}$-coalgebra structure such that $H$ becomes a $C_{\infty}$-model for $X$ and let $L$ be $\pi_*(Y)$ with an $SL_{\infty}$-structure such that $L$ is an $SL_{\infty}$-model for $Y$. If we turn $H$ into an $B_{op} \Omega_{op} \mathcal{COCOM}$ model using the morphism $\phi$ we can form the convolution algebra $Hom_{\K}(H,L)$. By Theorem \ref{thrmmappingmodel} this is a model for $Map_*(X,Y_{\Q})$, which proves the corollary. 
\end{proof}

\section{Application of the algebraic Hopf invariants: Rational homotopy theory}

In this section we discus how we can apply the algebraic Hopf invariants to rational homotopy theory and connect them to the work of Sinha and Walter in \cite{SW2}. In \cite{SW2}, Sinha and Walter study spaces of maps from the $n$-dimensional sphere to a rational target space $Y_{\Q}$. One of their main results is that they give a complete invariant of maps from the sphere to a rational space $Y_\Q$. In this section we will show how we recover their work by applying the algebraic Hopf invariants to $Map_*(S^n,Y_{\Q})$ and how to generalize this to general mapping spaces. 



The starting point is to use the functor from Theorem \ref{thrmQuillenstheorem} and apply this to the spaces $X$ and $Y$. In particular, two maps $f,g:X \rightarrow Y$ between $X$ and $Y$ are rationally homotopic if and only if $\mathcal{C} \lambda (f):\mathcal{C} \lambda (X) \rightarrow \mathcal{C} \lambda (Y)$ and $\mathcal{C} \lambda (g):\mathcal{C} \lambda (X) \rightarrow \mathcal{C} \lambda (Y)$ are homotopic as cocommutative coalgebra maps. 

When we apply this to maps from the sphere $S^n$ to a rational target space $Y_\Q$ of finite $\Q$-type, we get a map 
$$Map_*(S^n,Y_\Q) \rightarrow Hom_{CDGC}(\mathcal{C} \lambda (S^{n}),\mathcal{C} \lambda (Y_\Q)) \xrightarrow{mc_{\infty}} \mathcal{MC}(S^n,Y_\Q),$$
from the space of maps between $S^n$ and $Y_\Q$ to the moduli space of Maurer-Cartan elements of $Hom_{\K}(\mathcal{C} \lambda (S^{n}),\Omega_{\iota} \mathcal{C} \lambda (Y_\Q))$, which for simplicity we will denote by $\mathcal{MC}(S^n,Y_\Q)$. By Theorem \ref{thrmcompletehopfinvariants}, this map is a complete invariant of homotopy classes of maps, i.e. the map $mc_{\infty}:[S^n,Y_\Q]\rightarrow \mathcal{MC}(S^n,Y_\Q)$ is a bijection. Using this we can give an alternative completely algebraic proof of Theorem 2.10 in \cite{SW2}.

\begin{theorem}
 Let $Y_\Q$ be a 1-reduced rational space of finite $\Q$-type,  then there is an isomorphism of groups between $mc_{\infty}:\pi_n(Y_\Q) \rightarrow H_n( \Omega_{\iota} \mathcal{C} \lambda (Y_\Q))$.
\end{theorem}

\begin{proof}
To prove the theorem we want to use the completeness of the algebraic Hopf invariants, to do this we will first need to compute the moduli space of Maurer-Cartan elements. The moduli space of Maurer-Cartan elements is independent of the models chosen for $S^n$ and $Y_\Q$, since the set of homotopy classes of maps between $S^n$ and $Y_\Q$ is independent of the models chosen (as long as the models are of good enough quality, i.e. the model for $S^n$ should be cofibrant and the model for  $Y_\Q$ should be fibrant). Note that the Hopf invariant does depend on the models chosen and in particular on the choice of the maps $i$ and $p$.

To compute the moduli space of Maurer-Cartan elements we pick for $S^n$ the CDGC model given by a $1$-dimensional vector space in degree $n$ with the trivial comultiplicative structure, by abuse of notation we also denote this by $S^n$. A CDGC model for $Y_\Q$ will also by abuse of notation denoted by $Y_\Q$. As is shown in Chapter 22 of \cite{FHT}, a Lie model for $Y_\Q$ is then given by $\Omega_{\iota}Y_\Q$. The convolution $SL_{\infty}$ algebra between $S^n$ and $\Omega_{\iota} Y_\Q $ is now given by $Hom_{\K}(S^n,\Omega_{\iota} Y_\Q)$. Since $S^n$ is $1$-dimensional this is isomorphic as a chain complex to the $n$-fold desuspension of $Y_\Q$. Because $\Delta_{S^n}=0$, all the brackets $l_n$ on the convolution algebra $Hom_{\K}(S^n ,\Omega_{\iota}  Y_\Q)$ are zero, except for $l_1$ which is the desuspended differential of $\Omega_{\iota} Y_\Q$. 

The Maurer-Cartan equation for an element $x \in Hom_{\K}(S^n , \Omega_{\iota} Y_\Q)$ is therefore given by $d(x)=0$ and the space of Maurer-Cartan elements can be identified with the space of cycles in $(Y_\Q)_n$. The gauge equivalence relation on the set of Maurer-Cartan elements is given by the relation of homology, i.e. two cycles are gauge equivalent if and only if they are homologous. So as sets $\mathcal{MC}(S^n,Y_\Q)$ is isomorphic to $H_n(\Omega_{\iota} Y_\Q)$. 

Because the Hopf invariants are a complete invariant, there is an explicit bijection of sets from $[S^n,Y_\Q]$ to  $H_n(\Omega_{\iota} Y_\Q)$ given by $mc_{\infty}:[S^n,Y_\Q] \rightarrow H_n (\Omega_{\iota} Y_\Q)$. So what is left  to show is that this is an isomorphism of groups. The group structure on $\pi_n(Y_\Q)$ is coming from the pinch map $S^n\rightarrow S^n \vee S^n$ which is algebraically  modeled by  the diagonal map $\delta: \Q \rightarrow \Q \oplus \Q$ $\delta (x)=x \oplus x$. The group structure is then given by $(f*g)(x)=f(x) +g(x)$ for two maps $f,g:S^n\rightarrow Y_\Q$, this is exactly the same as the group structure on  the moduli space of Maurer-Cartan elements $\mathcal{MC}(S^n,Y_\Q)$, which is given by addition of the elements.
\end{proof}

The algebraic Hopf invariants have the advantage that they generalize the work of Sinha and Walter to more general spaces. In particular they are a complete invariant of rational homotopy classes of maps.

\begin{theorem}
 Let $f,g:X \rightarrow Y$ be maps from a finite CW-complex $X$ to a space $Y$, then $f$ and $g$ are rationally homotopic if and only if they have the same algebraic Hopf invariant, i.e. $mc_{\infty}(f)=mc_{\infty}(g)$. 
\end{theorem}

\begin{proof}
 The theorem follows immediately from the completeness of Theorem \ref{thrmcompletehopfinvariants}.
\end{proof}

Although the theorem gives a complete invariant of homotopy classes of maps, it might be hard to compute the actual values of this invariant when we are using the functor $\C\lambda$. In \cite{Wie2}, we make these invariants more computable for smooth manifolds, by replacing the  functor $\C\lambda$ by the de Rham complex. By doing this we can compute the map $mc$ by computing certain integrals. In that paper we also give some techniques for how to obtain information about the moduli space of Maurer-Cartan elements.


\bibliographystyle{plain}

\bibliography{bibliography}{}

\end{document}